\documentclass[12pt,psamsfonts, draft]{amsart}
\usepackage{amssymb,amsmath}
\usepackage{amsthm}
\usepackage{mathrsfs}
\usepackage{enumerate, enumitem, indentfirst}
\usepackage[fleqn,tbtags]{mathtools}
\usepackage{color}
\usepackage{courier}
\usepackage{hyperref}
\usepackage[a4paper,left=2.1cm,right=2.1cm,top=3.2cm,bottom=3.2cm]{geometry}
\usepackage{multicol}
\usepackage{xr}

\usepackage{multicol}
\usepackage{geometry}
\usepackage{marginnote}

 \newtheorem{theorem}{Theorem}
 \newtheorem{corollary}[theorem]{Corollary}
 \newtheorem{lemma}[theorem]{Lemma}
 \newtheorem{proposition}[theorem]{Proposition}

\newcommand{\dupN}{\mathbb{N}}

\newcommand{\seq}[1]{(#1_{n})_{n\in\dupN}}

\newcommand{\nen}{n\in\mathbb{N}}

\DeclareMathOperator{\ex}{ex}

\DeclarePairedDelimiterX\sip[2]{(}{)}{#1\,\delimsize\vert\,#2}
\DeclarePairedDelimiterX\sipa[2]{\langle}{\rangle_{\!_{A}}}{#1\,\delimsize\vert\,#2}
\DeclarePairedDelimiterX\sipb[2]{\langle}{\rangle_{\!_{B}}}{#1\,\delimsize\vert\,#2}
\DeclarePairedDelimiterX\siptilde[2]{(}{)_{\!_{\widetilde{A}}}}{#1\,\delimsize\vert\,#2}
\DeclarePairedDelimiterX\sipw[2]{(}{)_{\mathfrak{w}}}{#1\,\delimsize\vert\,#2}
\DeclarePairedDelimiterX\sipg[2]{(}{)_{g}}{#1\,\delimsize\vert\,#2}
\DeclarePairedDelimiterX\sipt[2]{(}{)_{\mathfrak{t}}}{#1\,\delimsize\vert\,#2}
\DeclarePairedDelimiterX\set[2]{\{}{\}}{#1\,\delimsize\vert\,#2}
\DeclarePairedDelimiterX\dual[2]{\langle}{\rangle}{#1,#2}

\DeclareMathOperator{\dom}{dom}
\DeclareMathOperator{\ran}{ran}

\newcommand{\projb}{\mathbf{P}(\hilb)}
\newcommand{\proj}{\mathbf{P}(\hil)}

\newcommand{\D}{\mathfrak{X}}

\newcommand{\hil}{\mathcal{H}}

\newcommand{\hila}{\hil_A}
\newcommand{\hilb}{\hil_B}

\newcommand{\bh}{\mathbf{B}(\hil)}

\newcommand{\bph}{\mathbf{B}_+(\mathcal{H})}

\newcommand{\xf}{\mathfrak{x}}

\newcommand{\dtw} {\mathbf{D}_{\tf}\wf}
\newcommand{\dwt}{\mathbf{D}_{\wf}\tf}

\newcommand{\uf}{\mathfrak{u}}
\newcommand{\vf}{\mathfrak{v}}
\newcommand{\tf} {\mathfrak{t}}
\newcommand{\wf} {\mathfrak{w}}
\newcommand{\ssf} {\mathfrak{s}}
\newcommand{\Xx}{\mathfrak{X}}

\newcommand{\nullf}{\mathfrak{0}}

\newcommand{\fpx}{\mathcal{F}_+(\Xx)}
\renewcommand*{\phi}{\varphi}

\DeclareMathOperator{\tform}{\mathfrak{t}}
\DeclareMathOperator{\wform}{\mathfrak{w}}

\newcommand{\hilw}{\hil_{\wform}}
\newcommand{\hilt}{\hil_{\tform}}

\newcommand{\lk}{\lambda_k}
\newcommand{\lkp}{\lambda_{k+1}}
\newcommand{\dst}{\mathbf{D}_{\ssf}\tf}
\newcommand{\dts}{\mathbf{D}_{\tf}\ssf}

\renewcommand*{\phi}{\varphi}

\newcommand{\kert}{\ker\tf}

\DeclareMathOperator{\quasi}{Q}

\begin{document}
\title{Quasi-units as orthogonal projections}

\author[Zsigmond Tarcsay]{Zsigmond Tarcsay}
\address{Zsigmond Tarcsay, Department of Applied Analysis\\ E\"otv\"os Lor\'and University\\ P\'azm\'any P\'eter s\'et\'any 1/c.\\ Budapest H-1117\\ Hungary}

\email{tarcsay@cs.elte.hu}
\author[Tam\'as Titkos]{Tam\'as Titkos}
\address{Tam\'as Titkos, Alfr\'ed R\'enyi Institute of Mathematics\\ Hungarian Academy of Sciences\\ Re\'altanoda u. 13-15.\\
Budapest H-1053\\ Hungary\\ and BBS University of Applied Sciences\\ Alkotm\'any u. 9.\\
Budapest H-1054\\ Hungary}

\email{titkos.tamas@renyi.mta.hu}

\thanks{Corresponding author: Tam\'as Titkos (titkos.tamas@renyi.mta.hu)}
\thanks{Zsigmond Tarcsay was supported by the Hungarian Ministry of Human
Capacities, NTP-NFT\"O-17. }
\subjclass{Primary 47A07 Secondary 47B65}

\keywords{Quasi-unit; Orthogonal projection; Extreme points; Galois connection.}

\begin{abstract} The notion of quasi-unit has been introduced by Yosida in unital Riesz spaces. Later on, a fruitful potential theoretic generalization was obtained by Arsove and Leutwiler. Due to the work of Eriksson and Leutwiler, this notion also turned out to be an effective tool by investigating the extreme structure of operator segments. This paper has multiple purposes which are interwoven, and are intended to be equally important. On the one hand, we  identify quasi-units as orthogonal projections acting on an appropriate auxiliary Hilbert space. As projections form a lattice and are extremal points of the effect algebra, we conclude the same properties for quasi-units. Our second aim is to apply these results for nonnegative sesquilinear forms. Constructing an order preserving bijection between operator- and form segments, we provide a characterization of being  extremal in the convexity sense, and we give a necessary and sufficient condition for the existence of the greatest lower bound  of two forms. Closing the paper we revisit some statements by using the machinery developed by Hassi, Sebesty\'en, and de Snoo. It will turn out that quasi-units are exactly the closed elements with respect to the antitone Galois connection induced by parallel addition and subtraction.
\end{abstract}

\maketitle

\section{Introduction} 
The notion of quasi-unit has been introduced by Yosida in \cite{Yosida} by developing an abstract Radon-Nikodym theorem in vector lattices. Originally, a nonnegative element $E$ of a Riesz space with identity $1$ was called a
quasi-unit if it satisfies $E\wedge (1-E)=0$. This property was systematically investigated earlier by Freudenthal in \cite{Freudenthal} and by Riesz in \cite{Riesz} who called such an element as a disjoint part of $1$. The importance of this concept in abstract potential theory was recognized by Arsove and Leutwiler in \cite{AL}. Later on, this notion turned out to be an effective tool by investigating the partial order of positive operators acting on a complex Hilbert space. Namely, quasi-units are characterized as extreme points of operator segments \cite{EL}. As the set of bounded positive operators is not a lattice, the definition of being a quasi-unit is much more difficult. One possible treatment involves the theory of Lebesgue-type decompositions developed by Ando in \cite{Ando-LD}. For this reason, we begin Section \ref{S: operators} with a short overview of the corresponding notions and results. After that, we present a factorization method based on \cite{Tarcsay_lebdec}, which allows us to identify quasi-units as orthogonal projections acting on an appropriate auxiliary Hilbert space. As projections form a lattice and are extremal points of the effect algebra, we obtain the same properties for quasi-units.

Our second aim is to carry over these results  for nonnegative sesquilinear forms. Namely, to investigate the order structure of nonnegative sesquilinear forms, and to prove that the set of quasi-units is a lattice. This will be done in Section \ref{S: forms}. In order to buttress the relevance of such a generalization, we mention some essentially different structures where nonnegative sesquilinear forms arise in a very natural way. For example, every nonnegative finite measure induces a form (i.e., a semi inner product) on the vector space of simple functions in the usual way. It is known that the set of nonnegative finite measures is a lattice (see \cite[\textsection 36]{ab}) in contrast to the set of nonnegative forms, as we shall see. That means that the infimum  of two induced forms may not exist among forms, although it does exist among forms induced by measures.  We encounter the same difficulty when investigating the order structure of representable functionals on a ${}^*$-algebra: it may happen that not every nonnegative form is induced by a representable functional. That is the reason why the set of forms and the set of functionals may carry completely different order structures. To see both positive and negative examples  we refer the reader to  \cite[Example 5.3 and 5.4]{glasgow}).

In Section \ref{S: paps}, we will make some general comments regarding parallel sum and difference of forms. It will turn out that quasi-units are precisely the closed elements with respect to the antitone Galois connection induced by parallel addition. We will also show how some results of this paper could be proved by using the machinery developed by Hassi, Sebesty\'en, and de Snoo in \cite{lebdec}, and by mimicking the purely algebraic presentation of the subject presented by Eriksson and Leutwiler in \cite{EL}.

\section{Positive operators}\label{S: operators}
Let $\hil$ be a complex Hilbert space and denote by $\bh$ the $C^*$-algebra of bounded linear operators on $\hil$ and by $\proj$ the lattice of orthogonal projections. A bounded operator $A$ is called positive ($A\in\bph$, in symbols) if its quadratic form is positive semi-definite, i.e.,
\begin{align*}
\sip{Ax}{x}\geq0,\qquad x\in\hil.
\end{align*}
According to Kadison's famous result \cite{kadison}, the set of self-adjoint operators with this partial order is an ``anti-lattice''. It means that the greatest lower bound of two self-adjoint operators $A$ and $B$ exists if and only if they are comparable, that is,
\begin{align*}
A\leq B\qquad\mbox{or}\qquad B\leq A.
\end{align*} The situation in the positive cone is much more difficult. The infimum problem in $\bph$ was a long standing open problem, which was solved first in the finite dimensional case by Gudder and Moreland in \cite{GM}. In the general case, a necessary and sufficient condition was established by Ando in \cite{andoinf}. Although the answer is that $\bph$ is far from being a lattice, a crucial role in the characterization has been played by a ``minimum-like'' operation called parallel addition, and an operator called the generalized short. The first part of this section is devoted to collect all the necessary notions and to offer several characterizations of the generalized short.

In order to present Ando's results, we have to recall some definitions. Let us fix two operators $A,B\in\bph$. Singularity of $A$ and $B$ ($A\perp B$, in symbols) means that $$\ran A^{1/2}\cap\ran B^{1/2}=\{0\},$$ or equivalently, $C\leq A$ and $C\leq B$ imply $C=0$ for all $C\in\bph$. We say that the operator $B$ is $A$-absolutely continuous ($B\ll A$, in notation) if there exists a monotone increasing sequence $(B_n)_{n\in\mathbb{N}}$ in $\bph$ whose strong limit is $B$, and each member satisfies $B_n\leq c_n A$ with some suitable constant $c_n\geq0$. A nontrivial fact is that absolute continuity is equivalent with closability (see \cite[Theorem 3]{Tarcsay_lebdec}). We say that $B$ is $A$-closable if
\begin{align*}
\sip{B(x_n-x_m)}{x_n-x_m}\to0\quad\mbox{and}\quad\sip{Ax_n}{x_n}\to0\qquad\Longrightarrow\qquad\sip{Bx_n}{x_n}\to0
\end{align*}
holds for all sequence $(x_n)_{n\in\mathbb{N}}$. Both absolute continuity and singularity can be characterized by means of parallel addition. The parallel sum of $A$ and $B$ is the unique operator (denoted by $A:B$) whose quadratic form is
$$\sip{(A:B)x}{x}=\inf\set{\sip{A(x-y)}{x-y}+\sip{By}{y}}{y\in\hil},\qquad x\in\hil.$$
For the properties of parallel addition we refer the reader to Anderson and Duffin \cite{sapa}, to Pekarev and {\v{S}}mul'jan \cite{PekarevSmuljan}, and to Erikson and Leutwiler \cite{EL}. According to \cite[Lemma 1]{Ando-LD}, and \cite[Theorem 11]{AT} we can describe absolute continuity and singularity as follows:
\begin{align*}
B\ll A\quad\Longleftrightarrow\quad B=\sup_{n\in\mathbb{N}}B:(nA)=:[A]B\qquad\mbox{and}\qquad A\perp B\quad\Longleftrightarrow A:B=0.
\end{align*}

The operator $[A]B$ is called the generalized short of $B$ with respect to (the range of) $A$. The remaining part of this section is mainly about the theoretical importance and about various descriptions of this object. The first result is the analogue of the well known Lebesgue decomposition of   measures.

\begin{theorem}\label{T: LDT}
For every positive operators $A$ and $B$ the decomposition
$$B=[A]B+(B-[A]B)$$
is of Lebesgue-type, that is,  $[A]B$ is $A$-absolutely continuous and $B-[A]B$ is $A$-singular. Moreover, $[A]B$ is the largest element of the set $$\set{C\in\bph}{C\leq B,~C\ll A}.$$ This type of decomposition is unique if and only if $[A]B\leq\alpha A$ holds for some $\alpha\geq0$.
\end{theorem}
For the proof we refer the reader to \cite[Theorem 2 and Theorem 6]{Ando-LD}. (See also \cite{Arlinskii-ps} and \cite{Tarcsay_lebdec}.)
Analogous results hold true for many different objects such as representable functionals \cite[Theorem 3.3]{Tarcsay_repr}, and positive definite operator functions \cite[(2.8) Theorem]{as}.
As it was mentioned earlier, Ando proved in \cite[Theorem 4]{andoinf} that the generalized short also plays a crucial role by deciding whether the infimum (or greatest lower bound) of two operator exists.
\begin{theorem}\label{T: inf}
Given $A,B\in\bph$, the infimum $A\wedge B$ exists in $\bph$ if and only if $[A]B$ and $[B]A$ are comparable. In this case
\begin{align*}
A\wedge B=\min\{[A]B,[B]A\}.
\end{align*}
\end{theorem}
Using the necessary and sufficient condition in Ando's infimum theorem, it is easy to find pairs of operators for which the infimum does not exists. However, it is known that the set of orthogonal projections is a lattice. As a matter of fact, the infimum of two projections $P$ and $Q$ in the projection lattice is the same as the infimum of $P$ and $Q$ in $\bph$, and is equal to $2(P:Q)$. We claim that the generalized shorts of $B$ behave similar as projections. In fact, they are indeed projections acting on an appropriate Hilbert space.
To be more precise and prove that claim we recall the construction from \cite[Theorem 1]{Tarcsay_lebdec}. Let us consider the range space $\ran A$ of $A$ endowed with the inner product 
\begin{align*}
\sipa{Ax}{Ay}=\sip{Ax}{y},\qquad x,y\in\hil.
\end{align*}
Note that $\sip{Ax}{x}=0$ implies $Ax=0$, hence $\sipa{\cdot}{\cdot}$ is an inner product, indeed. Let $\hila$ denote the completion of that pre-Hilbert space. The canonical embedding operator $J_A$ of $\hila$ into $\hil$, is defined by 
\begin{align*}
J_A(Ax)=Ax,\qquad x\in\hil.
\end{align*}
Since it is defined on the dense subspace $\ran A$ of $\hila$, it is continuous according to the Schwarz inequality
\begin{align*}
\|Ax\|^2\leq \|A\|\cdot \sip{Ax}{x},\qquad x\in\hil.
\end{align*}
We continue to write $J_A$ for the norm preserving extension to $\hila$. An immediate calculation shows that the adjoint operator $J_A^*$ is an operator from $\hil$ into $\hila$ satisfying 
\begin{align*}
J_A^*x=Ax,\qquad x\in\hil.
\end{align*}
Putting these two observations together, we obtain the following useful factorization of $A$:
\begin{align*}
J_AJ_A^*=A.
\end{align*}
Since $\ran J_A^*$ is identical with the dense subspace $\ran A\subseteq\hila$, we conclude that $J_A$ is injective.

Let us consider now another positive operator $B$ on $\hil$ with corresponding associated Hilbert space and operator  $\hilb$ and $J_B$, respectively. In order to phrase absolute continuity (or closability, which suits better to this consideration) in terms of these auxiliary Hilbert spaces and factorizations, we introduce the relation $\widehat{B}$ between $\hila$ and $\hilb$ as 
\begin{align*}
\widehat{B}:=\overline{\set{(Ax,Bx)}{x\in \hil}}\subseteq \hila\times\hilb.
\end{align*}
We remark that this relation is not single valued unless $\ker A\subseteq \ker B$. Let us denote by $\mathcal{M}$ the multivalued part of $\widehat{B}$, which is the closed linear subspace  
\begin{align*}
\mathcal{M}:=\set{\xi\in\hilb}{(0,\xi)\in\widehat{B}}.
\end{align*}
Setting $P_\mathcal{M}$ being the orthogonal projection of $\hilb$ onto the subspace $\mathcal{M}$, we obtain the following decomposition of  $B$: 
\begin{align*}
B=J_B(I-P_\mathcal{M})J_B^*+J_BP_\mathcal{M}J_B^*,
\end{align*}
where $J_B(I-P_\mathcal{M})J_B^*$ is closable with respect to $A$, and $J_BPJ_B$ and $A$ are singular (see \cite[Theorem 1]{Tarcsay_lebdec}). Furthermore, $J_B(I-P_\mathcal{M})J_B^*$ is maximal among the set of those positive operators $C\leq B$ which are closable with respect to $A$. 
Since  the properties of being closable and of being absolutely continuous are equivalent by \cite[Theorem 3]{Tarcsay_lebdec}, we  conclude that 
\begin{align}\label{E:[A]B=JPJ}
[A]B=J_B(I-P_\mathcal{M})J_B^*.
\end{align}
Using this factorization, we can characterize the extreme points of operator segments as follows. (For earlier version of this theorem see \cite[Theorem 4.3 and Theorem 4.7]{EL}.) For positive operators $S$ and $T$ let us denote the operator interval spanned by $S$ and $T$ by $[S,T]$, i.e.,
$$[S,T]=\set{A\in\bph}{S\leq A\leq T}.$$ \begin{theorem}\label{T:quasiunit-operator} Let $A,B$ be bounded positive operators in the Hilbert space $\hil$ such that $A\leq B$. The following assertions are equivalent: 
\begin{enumerate}[label=\textup{(\roman*)}, labelindent=\parindent]
\item  $A=[A]B$,
\item there is an orthogonal projection $P\in\projb$ such that $A=J_BPJ_B^*$,
\item $A$ is an extreme point of the interval $[0,B]$,
\item $A$ and $B-A$ are singular.
\end{enumerate}
\end{theorem}
\begin{proof}
Implication (i)$\Rightarrow$(ii) is clear by \eqref{E:[A]B=JPJ}. To  prove (ii)$\Rightarrow$(iii) we show for a given $P\in\projb$ that $A=J_BPJ_B$ is an extreme point of $[0,B]$. To this end assume that $A=\alpha_1A_1+\alpha_2A_2$ for some $0\leq A_i\leq B$ and $\alpha_i>0$, $\alpha_1+\alpha_2=1$. Since the correspondences 
\begin{align*}
J_B^*x\mapsto A_i^{1/2}x,\qquad x\in \hil,
\end{align*}
are linear contractions, it is readily seen  that there exist two positive contractions $\widetilde{A}_1,\widetilde{A}_2$ on $\hilb$ such that $A_i=J_B\widetilde{A}_iJ_B^*$. Since $J_B$ is injective and $J_B^*$ has dense range, we conclude that $P=\alpha_1 \widetilde{A}_1+\alpha_2\widetilde{A}_2$ and therefore that $ \widetilde{A}_1=\widetilde{A}_2=P$ because $P$ is an extreme point of $[0,I]$. A very similar argument shows that (iii) implies (ii). Next we are going to prove (ii)$\Rightarrow$(iv). Suppose $A=J_BPJ_B^*$ for some orthogonal projection $P$ of $\hilb$, then $B-A=J_B(I-P)J_B^*$. By injectivity of $J_B$ we conclude that 
\begin{align*}
\ran A^{1/2}\cap\ran (B-A)^{1/2}=\ran J_BP\cap \ran J_B(I-P)=\{0\},
\end{align*}
hence $A$ and $B-A$ are singular. Finally, to prove that (iv) implies (i), assume that $A$ and $B-A$ are singular and let $Q$ be a positive contraction on $\hilb$ such that $A=J_BQJ_B^*$. It is clear  that $Q$ and $I-Q$ are singular too. From the inequalities $Q^2\leq Q$ and $(I-Q)^2\leq I-Q$ it follows that 
\begin{align*}
Q(I-Q)\in[0,Q]\cap[0,I-Q],
\end{align*}
hence $Q=Q^2$ by singularity. That means that $Q$ is an orthogonal projection. The proof will be done if we show that $Q=I-P_{\mathcal{M}}$. It is clear that $A+(B-A)$ is a Lebesgue decomposition of $B$ with respect to $A$, hence from the maximality of $J_B(I-P_{\mathcal{M}})J_B^*$ we infer that $A\leq J_B(I-P_{\mathcal{M}})J_B^*$, and hence $Q\leq I-P_{\mathcal{M}}$. If $Q\neq I-P_{\mathcal{M}}$, then we can find a non-zero vector  $\xi\in\ran (I-P_{\mathcal{M}})\cap[\ran Q]^{\perp}$. By density, there is a sequence $(x_n)_{n\in\dupN}$ from $\hil$ such that $J_B^*x_n\to\xi$. Then we have 
\begin{align*}
\sip{Ax_n}{x_n}=\sipb{QJ_B^*x_n}{J_B^*x_n}\to0,
\end{align*}
and similarly,
\begin{align*}
\sip{[A]B(x_n-x_m)}{x_n-x_m}=\sipb{(I-P_{\mathcal{M}})J_B^*(x_n-x_m)}{J_B^*(x_n-x_m)}\to0,
\end{align*}
but 
\begin{align*}
\sip{[A]Bx_n}{x_n}=\sipb{(I-P_{\mathcal{M}})J_B^*x_n}{J_B^*x_n}\to \sipb{\xi}{\xi}\neq0,
\end{align*}
which contradicts the absolute continuity of $[A]B$. The proof is therefore complete. 
\end{proof}
\begin{corollary}\label{corollary [A]B quasi-unit} 
Let $B\in\bph$ be a positive operator. Then $[A]B$ is an extreme point of $[0,B]$ for all $A\in\bph$. Moreover, $A+[A]B$ is an extreme point of $[0,A+B]$.

\end{corollary}
\begin{proof} The first statement follows immediately from the previous theorem, because $[A]B$ is represented by a projection according to \eqref{E:[A]B=JPJ}. To prove the second assertion it is enough to show that $B-[A]B$ and $A+[A]B$ are  singular. We recall that singularity of positive operators is characterized by having zero parallel sum. Let us therefore calculate the quadratic form of the parallel sum of $B-[A]B$ and $A+[A]B$
\begin{align*}
&\sip{(B-[A]B):(A+[A]B)x}{x}=\\
&=\inf_{y\in\hil} \Big\{\sip{J_B(I-P_{\mathcal{M}})J_B^*(x-y)}{x-y}+\sip{A(x-y)}{x-y}+ \sip{J_BP_{\mathcal{M}}J_B^*y}{y} \Big\}\\
&=\inf_{y\in\hil} \Big\{ \sipa{A(x-y)}{A(x-y)} + \sipb{(I-P_{\mathcal{M}})(Bx-By)}{Bx-By}+\sipb{P_{\mathcal{M}}(By)}{By}\Big\}\\
&=\inf_{y\in\hil} \Big\{ \sipa{A(x-y)}{A(x-y)} +\sipb{(I-P_\mathcal{M})(Bx)-By}{(I-P_\mathcal{M})(Bx)-By} \Big\}
\end{align*}
where the last term is just the squared distance of  vector $(Ax,(I-P)Bx)$ from the dense set $\ran A\times \ran B$ in $\hila\times\hilb$, hence is equal to zero. 
\end{proof}

An operator $A\in[0,B]$ that satisfies the equivalent conditions of Theorem \ref{T:quasiunit-operator} is called a $B$-quasi-unit and the set of $B$-quasi-units is denoted by $\quasi(B)$. It follows immediately from Theorem \ref{T:quasiunit-operator} and Corollary \ref{corollary [A]B quasi-unit} that
\begin{align*}
\quasi(B)=\ex [0,B]=\set{[A]B}{A\in\bph}=\set{J_BPJ_B^*}{P\in \projb}.
\end{align*}
\begin{theorem}\label{T: op}
Let $B$ be a positive operator on the Hilbert space $\hil$. The  mapping 
\begin{equation}\label{E:PhiHB}
\Psi_B(P):=J_BPJ_B^*
\end{equation}
establishes an order preserving bijection between $\projb$ of $\quasi(B)$.
\end{theorem}
\begin{proof}
According to Theorem \ref{T:quasiunit-operator}, $\Psi_B$ is a surjection and since $J_B^*$ has dense range and $J_B$ is one-to-one, $\Psi_B$ is also an injection. It is immediate that $\Psi_B$ is order preserving. That $\Psi_B^{-1}$ is order preserving too follows again from the fact that the range of $J_B^*$ is dense.
\end{proof}
As a straightforward consequence of the fact that $\projb$ is a lattice, we obtain \cite[Theorem 5.4]{EL} as an easy corollary of Theorem \ref{T: op}.
\begin{corollary}\label{Q(B) lattice}
The partially ordered set $(\quasi(B),\leq)$ is a lattice.
\end{corollary}

\section{Nonnegative sesquilinear forms}\label{S: forms}
The aim of this section is to carry over the previous results  for nonnegative sesquilinear forms. Namely, to determine the extreme points of order segments and to decide whether or not the greatest lower bound of two nonnegative sesquilinear forms exists. To begin with, let us recall some basic notions. A nonnegative sesquilinear form on the complex linear space $\D$ is a mapping $\tf:\D\times\D\to\mathbb{C}$, which is linear in its first, anti-linear in its second argument, and the corresponding quadratic form is nonnegative (that is, $\tf[x]:=\tf(x,x)\geq0$ for all $x\in\D$). Since all the sesquilinear forms in this paper are assumed to be nonnegative, we will say shortly \emph{form}. Since the square root of the quadratic form is a seminorm, the set $\kert=\set{x\in\Xx}{\tf[x]=0}$ is a linear subspace of $\Xx$. The quotient space $\Xx/_{\kert}$ endowed with $\sip{x+\kert}{y+\kert}_{\tf}:=\tf(x,y)$ is a pre-Hilbert space. Let us denote its completion by $\hilt$. The cone of forms will be denoted by $\fpx$. We write $\tf\leq\wf$ if $\tf[x]\leq\wf[x]$ holds for all $x\in\D$. If $\tf$ and $\wf$ are forms such that $\tf\leq\wf$, then the interval (or segment) spanned by $\tf$ and $\wf$ is
$$[\tf,\wf]=\set{\ssf\in\fpx}{\tf\leq\ssf\leq\wf}.$$
A sequence $(\tf_n)_{\nen}$ of forms is monotone increasing if $m\leq n$ implies that $\tf_m\leq\tf_n$. The pointwise limit of an upper bounded monotone increasing sequence ($m\leq n$ implies $\tf_m\leq\tf_n\leq\ssf$ for some form $\ssf$) is a form, and $\tf=\lim\limits_{n\to\infty}\tf_n=\sup\limits_{n\in\mathbb{N}}\tf_n\leq\ssf$. The form $\tf$ is $\wf$-dominated if there exists an $\alpha\geq0$ such that $\tf\leq \alpha\wf$. The form $\tf$ is $\wf$-almost dominated if $\tf$ is the (order) supremum of some sequence $\seq{\tf}$ of $\wf$-dominated forms. The forms $\tf$ and $\wf$ are singular if for each form $\ssf$ on $\D$ the inequalities $\ssf\leq\tf$ and $\ssf\leq\wf$ imply that $\ssf=\mathfrak{0}$ (where $\nullf$ denotes the zero form). A decomposition of $\tf$ into $\wf$-almost dominated and $\wf$-singular parts is called $\wf$-Lebesgue decomposition. Exactly as in the bounded positive operator case, parallel addition plays a crucial role here. This notion was introduced by Hassi, Sebesty\'en, and de Snoo in their fundamental paper \cite{lebdec}. The definition and the properties of parallel addition are given in the following proposition  (for the details see \cite[Proposition 2.2 and Lemma 2.3]{lebdec}).

\begin{proposition}\label{parsumproperty}
For $\tf,\wf\in\fpx$ the quadratic form of the parallel sum $\tf:\wf$ is
    \begin{align}\label{ps def}
        (\tf:\wf)[x]:=\inf\limits_{y\in\D}\big\{\tf[x-y]+\wf[y]\big\},\qquad x\in\D.
    \end{align}
Furthermore, parallel addition satisfies the following properties
\begin{enumerate}[label=\textup{(\roman*)}, labelindent=\parindent]
\begin{multicols}{2}
\item $\tf:\wf=\wf:\tf\leq\tf$,
\item $(\tf:\wf):\ssf=\tf:(\wf:\ssf)$,
\item $\tf\leq\ssf~\Rightarrow~\tf:\wf\leq\ssf:\wf$,
\item $\lambda\tf:\mu\tf=\textstyle{\frac{\mu\lambda}{\mu+\lambda}}\tf$\quad$(\lambda,\mu>0)$.
\end{multicols}
\end{enumerate}
\end{proposition} 
Since parallel addition is monotone in both variables according to Proposition \ref{parsumproperty} (i) and (iii), the sequence $(\tf:n\wf)_{n\in\mathbb{N}}$ is monotone increasing and is bounded above by $\tf$. Thus the equality $\dwt:=\sup\limits_{n\in\mathbb{N}}\tf:n\wf$ defines a form which satisfies
\begin{align*}
\tf:n\wf\leq\dwt\leq\tf\qquad\mbox{for all}~n\in\mathbb{N}.
\end{align*} 
It is obvious that this form will play the role of the generalized short $[A]B$ in this setting. In analogy with the operator case, almost domination and singularity can be characterized as follows: $\tf$ is $\wf$-almost dominated if and only if $\dwt=\tf$ (see \cite[Theorem 2.6]{lebdec}), $\tf$ and $\wf$ are singular if $\tf:\wf=\nullf$ (see \cite[Proposition 2.10]{lebdec}).

In the next theorem we are going to offer a construction (see also \cite{titkosinf}) which allows us to carry over the results of Section \ref{S: operators} for nonnegative sesquilinear forms. As it will be discussed later in Section \ref{S: paps}, the following characterization of quasi-units could be obtained by using only the properties of parallel addition and subtraction, or by using the notion of complement as in \cite{cof}. Recall that if $\hil$ is a Hilbert space then there is a bijective isometric correspondence between operators in $\bh$ and bounded
sesquilinear forms on $\hil$, given by $A\mapsto \tf_A$, where $A\in \bh$ and
\begin{align*}
\tf_A(x,y)=\sip{Ax}{y}\qquad x,y\in \hil.
\end{align*}
\begin{theorem}\label{T: S3}
Let $\tf\in\fpx$ be an arbitrary fixed form. Then there is a Hilbert space $\hilt$ such that the form interval $[\nullf,\tf]$ is order isomorphic with the operator interval $[0_{\tf},I_{\tf}]$ of $\mathbf{B}(\hilt)$. In particular, the following statements are equivalent for any form $\wf\in[\nullf,\tf]$:
\begin{enumerate}[label=\textup{(\roman*)}, labelindent=\parindent]
\item $\wf$ is a $\tf$-quasi-unit, i.e., $\dwt=\wf$,
\item $\wf$ is an extreme point of $[\nullf,\tf]$,
\item $\wf$ is a disjoint part of $\tf$.
\end{enumerate}
Furthermore, the infimum of $\uf$ and $\vf$ exists if and only if $\mathbf{D}_{\uf}\vf$ and $\mathbf{D}_{\vf}\uf$ are comparable.
\end{theorem}

\begin{proof}
First we are going to prove that such an order isomorphism exists. Let $\wf$ be a form such that $\wf\in[\nullf,\tf]$ and consider the auxiliary Hilbert spaces $\hilt$ and $\hilw$, respectively. The linear operator $j_{\tf,\wf}$ from $\D/_{\ker{\tf}}\subseteq
\hilt$ to $\hilw$ defined by
\begin{align*}
j_{\tf,\wf}(x+\ker\tf)=x+\ker\wf,\qquad x\in\D,
\end{align*}
is a densely defined contraction, hence its
closure (i.e. the second adjoint of $j_{\tf,\wf}$) is a contraction with $\dom j_{\tf,\wf}=\hilt$ such that $\ran j_{\tf,\wf}$ is dense in
$\hilw$. We continue to use the symbol $j_{\tf,\wf}$ for the closure instead of $j_{\tf,\wf}^{**}$. Now, observe that  $j_{\tf,\wf}^{*}j_{\tf,\wf}$ is a  bounded  positive operator on $\hilt$ and satisfies
\begin{align}\label{formop}
\begin{split}
\sipt[\big]{j_{\tf,\wf}^{*}j_{\tf,\wf}(x+\ker\tf)}{y+\ker\tf}&=\sipw[\big]{j_{\tf,\wf}(x+\ker\tf)}{j_{\tf,\wf}(y+\ker\tf}\\
&=\sipw{x+\ker\wf}{y+\ker\wf}\\
&=\wf(x,y).
\end{split}
\end{align}
We claim that the map
\begin{align*}
\Phi_{\tf}:[\nullf,\tf]\rightarrow [0_{\tf},I_{\tf}];\qquad \Phi_{\tf}(\wf)=j_{\tf,\wf}^{*}j_{\tf,\wf}
\end{align*}
is a bijection which preserves the order in both directions. Furthermore, it has the following properties:
\begin{itemize}
        \item[(a)] $\Phi_{\tf}$ preserves convex combinations,
        \item[(b)] $\Phi_{\tf}$ preserves parallel sum,
        \item[(c)] If $\uf$ and $\wf\in[\nullf,\tf]$, then $\mathbf{D}_{\uf}\wf \in[\nullf,\tf]$ and
        $\Phi_{\tf}(\mathbf{D}_{\uf}\wf)=[\Phi_{\tf}(\uf)]\Phi_{\tf}(\wf)$.
    \end{itemize} 
Indeed, it follows from the definition of $\Phi_{\tf}$ that if $\uf,\wf\in[\nullf,\tf]$, such that $\uf\neq\wf$ then we have 
\begin{align*}\sipt[\big]{j_{\tf,\uf}^{*}j_{\tf,\uf}(x_0+\ker\tf)}{x_0+\ker\tf}=\uf[x_0]\neq\vf[x_0]=\sipt[\big]{j_{\tf,\wf}^{*}j_{\tf,\wf}(x_0+\ker\tf)}{ x_0+\ker\tf}
\end{align*}
for some $x_0\in\D$, which means $\Phi_{\tf}(\uf)\neq\Phi_{\tf}(\vf)$. Conversely, let $A$ be an element of $[0_{\tf},I_{\tf}]$. Then the sesquilinear form  $\wf(x,y):=\sipt{A(x+\ker \tf)}{ y+\ker\tf}$ is in
$[\nullf,\tf]$, and
\begin{align*}
\sipt{j_{\tf,\wf}^{*}j_{\tf,\wf}(x+\ker\tf)}{y+\ker\tf}=\wf(x,y)=\sipt{A(x+\ker \tf)}{y+\ker\tf}
\end{align*}
implies that $\Phi_{\tf}(\wf)=A$, thus $\Phi_{\tf}$ is surjective. To prove (a), let $\uf,\vf\in[\nullf,\tf]$ and $\alpha,\beta>0$ such that $\alpha+\beta=1$. Since $\D/_{\ker{\tf}}$ is dense in $\hilt$, the equality 
$$\Phi_{\tf}(\alpha\uf+\beta\vf)=\alpha\Phi_{\tf}(\uf)+\beta\Phi_{\tf}(\vf)$$ 
can be obtained by the fact that the following two identity hold for all $x,y\in\Xx$:
\begin{align*}
(\alpha\uf+\beta\vf)(x,y)=\sipt{j_{\tf,\alpha\uf+\beta\vf}^{*}j_{\tf,\alpha\uf+\beta\vf}(x+\ker\tf)}{y+\ker\tf}
\end{align*}
and
\begin{align*}
\alpha\uf(x,y)+\beta\vf(x,y)=\alpha\sipt{j_{\tf,\uf}^{*}j_{\tf,\uf}(x+\ker\tf)}{y+\ker\tf}+\beta\sipt{j_{\tf,\vf}^{*}j_{\tf,\vf}(x+\ker\tf)}{ y+\ker\tf}.
\end{align*}
It follows from the definition of $\Phi_{\tf}$ that if $(\uf_n)_{n\in\mathbb{N}}$ is a monotone increasing sequence in $[\nullf,\tf]$ such that $\sup\limits_{n\in\mathbb{N}}\uf_n=\uf$, then
        $\uf\in[\nullf,\tf]$ and $\Phi_{\tf}(\uf)=\sup\limits_{n\in\mathbb{N}}\Phi_{\tf}(\uf_n)$. Consequently, property (b)  implies property (c). Again, since $\D/_{\ker{\tf}}$ is dense in $\hilt$, the parallel sum $\Phi_{\tf}(\uf):\Phi_{\tf}(\vf)$ can be computed as
        \begin{align*}
        \begin{split}
        &\sipt[\big]{(\Phi_{\tf}(\uf):\Phi_{\tf}(\vf))(x+\ker\tf)}{x+\ker\tf}=\\
        &\qquad=\inf_{y\in \D}\big\{\sipt[\big]{\Phi_{\tf}(\uf)(x-y+\ker\tf)}{ x-y+\ker\tf }+\sipt{\Phi_{\tf}(\vf)(y+\ker\tf)}{y+\ker\tf}\big\}
        \end{split}
        \end{align*}
        for all $x\in\Xx$. But this infimum is just the quadratic form of the parallel sum of $\uf$ and $\vf$ evaluated at point $x\in\Xx$. On the other hand, we have $\sipt{\Phi_{\tf}(\uf:\vf)(x+\ker\tf)}{x+\ker\tf}=(\uf:\vf)[x]$ for all $x\in\Xx$, by the definition of $\Phi_{\tf}$. This proves $\Phi_{\tf}(\uf:\vf)=\Phi_{\tf}(\uf):\Phi_{\tf}(\vf)$.

Now, equivalence of (i), (ii), and (iii) follows immediately from Theorem \ref{T:quasiunit-operator} and the properties of $\Phi_{\tf}$ just proved. Similarly, the criterion for the existence of the infimum of $\uf$ and $\vf$ is established from Theorem 2 along the map $\Phi_{\uf+\vf}$.
\end{proof}
The following corollary is an immediate consequence of Theorem \ref{T: S3} and Corollary \ref{Q(B) lattice}.
\begin{corollary} The set $\quasi(\tf)$ of $\tf$-quasi units is a lattice. That is, the greatest lower- and the least upper bound of any two $\tf$-quasi-units exists in $\quasi(\tf)$.
\end{corollary}
As we have seen, the greatest lower bound of two forms may not exist. We are going to show that the greatest lower bound of a $\tf$-quasi-unit with other elements of $[\nullf,\tf]$ always exists.

\begin{corollary}\label{C: S3} Let $\tf$ be a form on $\Xx$ and assume that $\wf$ is a $\tf$-quasi-unit. Then for every $\uf\in[\nullf,\tf]$ the infimum $\wf$ and $\uf$ exists, and $\wf\wedge\uf=\mathbf{D}_{\wf}\uf$.
\end{corollary}
\begin{proof} First observe that $\uf\leq\tf$ implies $\mathbf{D}_{\wf}\uf\leq\mathbf{D}_{\wf}\tf$, by definition. Thus we have the inequality $\mathbf{D}_{\wf}\uf\leq\wf$ according to the assumption that $\wf$ is a $\tf$-quasi-unit. Using that $\mathbf{D}_{\wf}\uf$ is $\uf$-almost dominated, we conclude $\mathbf{D}_{\wf}\uf=\mathbf{D}_{\uf}(\mathbf{D}_{\wf}\uf)\leq\mathbf{D}_{\uf}\wf$, or equivalently, $\wf\wedge\uf=\mathbf{D}_{\wf}\uf$.
\end{proof}
Since it will be used later, we mention separately a consequence of Corollary \ref{corollary [A]B quasi-unit}. Namely that $\dwt$ is a $\tf$-quasi-unit, or equivalently, a disjoint part of $\tf$ for all $\wf$, i.e.,
\begin{align}\label{F: dwt quasi-unit}
\dwt:(\tf-\dwt)=\nullf.
\end{align}
We close this section by making some remarks about applications. The results of Theorem \ref{T: S3} and Corollary \ref{C: S3} can be applied for measures, operator kernels, and representable functionals, but with careful consideration. As it was mentioned in the introduction, the necessary and sufficient condition for the existence of infimum may be only sufficient in the concrete situations, such as \cite[Example 5.3] {glasgow}. An interesting  area where the  results of this section directly appy, is the case of positive definite kernels. For the details see \cite[Section 7]{lebdec}. The main point there is \cite[Lemma 7.1]{lebdec} which says basically that there is a bijective correspondence between nonnegative operator kernels and nonnegative forms defined on an appropriate complex linear space.

\section{Parallel sum and parallel difference}\label{S: paps}
Due to the previous sections, the importance of parallel addition is obvious pertaining to the topic in consideration. In the first part of this section we are going to observe a new aspect, namely that absolute continuous elements are precisely those that are closed with respect to the closure operation induced by the parallel sum. To this end, first we recall some facts regarding parallel subtraction of forms from \cite{psonf}. To find some kind of inverse of parallel addition, it is natural to consider the equation 
\begin{align}\label{F: x:t=s}
\xf:\tf=\ssf
\end{align}
with given $\tf,\ssf\in\fpx$ and with unknown $\xf$. Although \eqref{F: x:t=s} is not always solvable, if it so, then it does exist a minimal solution according to \cite[Theorem 3.1]{psonf}. That minimal solution is 
\begin{align}\label{pd def}
(\ssf\div\tf)[x]=\sup_{y\in\Xx}\left\{\ssf[x+y]-\tf[y]\right\},\qquad x\in\Xx,
\end{align}
which is called the parallel difference of $\ssf$ and $\tf$. Observe that the equation $\xf:\wf=\tf:\wf$ is solvable for all $\tf$ and $\wf$, thus the form $(\tf:\wf)\div\wf$ always exists (which means, in addition, that the corresponding supremum is finite for every $x\in\Xx$). In fact, as it was proved by Hassi, Sebesty\'en, and de Snoo in \cite[Theorem 3.5 and Theorem 3.9]{lebdec}, this form is the largest $\wf$-closable (or $\wf$-almost dominated) part of $\tf$, i.e.,
\begin{align*}
\dwt=(\tf:\wf)\div\wf.
\end{align*}
An extremely important ingredient to prove this equality is the following line of identities
\begin{align}\label{F: t:w=dwt:w}
\tf:\wf=\dwt:\wf=\mathbf{D}_{\wf}(\tf:\wf),
\end{align}
and the crucial observation that the inequality $\tf:\wf\leq\ssf:\wf$ carries some information about the corresponding generalized shorts. More concretely, we have
\begin{align}\label{F: t:w<s:w iff Dwt<Dws}
\tf:\wf\leq\ssf:\wf\qquad\quad\mbox{if and only if}\quad\qquad\dwt\leq\mathbf{D}_{\wf}\ssf. 
\end{align}
The proof in \cite[Proposition 2.7]{lebdec} follows the algebraic argument of Eriksson and Leutwiler in \cite[Proposition 2.6 and Proposition 2.7]{EL}. This latter if and only if statement will be the key by the following simple observation. Let us recall what an antitone Galois connection is. Let $U_{{}_{\sqsubseteq}}:=(U,\sqsubseteq)$ and $V_{{}_{\leq}}:=(V,\leq)$ be two partially ordered sets. If $\alpha:U\to V$ and $\beta:V\to U$ are antitone (or order reversing) functions, such that
\begin{align}\label{gc prop}
v \leq \alpha(u)\quad\mbox{if and only if}\quad u\sqsubseteq \beta(v)\qquad\mbox{for all}~u\in U~\mbox{and}~v\in V,
\end{align}
then we say that the pair $(\alpha,\beta)$ establishes an antitone Galois connection
between $U_{{}_{\sqsubseteq}}$ and $V_{{}_{\leq}}$. 
It is easy to see that the property in \eqref{gc prop} is equivalent to
\begin{align*}
u\sqsubseteq\big(\beta\circ\alpha\big)(u)~\mbox{for all}~u\in U~\qquad\mbox{and}\qquad v\leq\big(\alpha\circ\beta\big)(v)~\mbox{for all}~v\in V.
\end{align*}

The functions $\alpha$ and $\beta$ are called polarities, the composition $\beta\circ\alpha:U \to U$ is called closure operator. The closure operator is monotone, idempotent, and extensive, that is, it satisfies 
\begin{align*}
u\sqsubseteq\big(\beta\circ\alpha\big)(u)\quad\mbox{for all}~u\in U.
\end{align*}
An element $u\in U$ is called $\beta\circ\alpha$-closed if $\big(\beta\circ\alpha\big)(u)=u$. Similarly, an element $v\in V$ is called $\alpha\circ\beta$-closed if $\big(\alpha\circ\beta\big)(v)=v$. Denoting the set of closed elements by $C_U$ and $C_V$, respectively, it is known that $\alpha\!\restriction_{C_U}$ and $\beta\!\restriction_{C_V}$ are order reversing bijections (or inverse isomorphisms) that are inverse to each other. For more details we refer the reader to \cite[Chapter  6, \S 11.1]{Kuros}. 

Now, let us fix a $\wf\in\fpx$, and consider the partially ordered sets $U_{{}_{\sqsubseteq}}:=\big(\fpx,\leq^{op}\big)$ and $V_{{}_{\leq}}:=\big(\fpx^{:\wf},\leq\big)$, where $\leq^{op}$ denotes the opposite order ($\uf\leq^{op}\vf\Leftrightarrow\vf\leq\uf$), and
\begin{align*}
\fpx^{:\wf}=\big\{\tf:\wf\,\big|\,\tf\in\fpx\big\}.\end{align*}
Then the maps $\alpha(\tf)=\tf:\wf$ and $\beta(\ssf)=\ssf\div\wf$ establish an antitone Galois connection between $U_{{}_{\sqsubseteq}}$ and $V_{{}_{\leq}}$. Indeed, $\alpha$ and $\beta$ are antitone, because of the definitions \eqref{ps def} and \eqref{pd def}. The only nontrivial fact is that these maps satisfy \eqref{gc prop}. But this is exactly the content of \eqref{F: t:w<s:w iff Dwt<Dws}. Indeed, let $\uf\in\fpx$ and $\vf:\wf\in\fpx^{:\wf}$ be arbitrary elements. If $\vf:\wf\leq\alpha(\uf)=\uf:\wf$ then $\mathbf{D}_{\wf}\vf\leq\mathbf{D}_{\wf}\uf\leq\uf$, which says \begin{align*}\uf\leq^{op}\mathbf{D}_{\wf}\vf=(\vf:\wf)\div\wf=\beta(\vf:\wf).
\end{align*}
For the converse implication assume that $\uf\leq^{op}\beta(\vf:\wf)$, that is, $\uf\geq \mathbf{D}_{\wf}\vf$. Since $\mathbf{D}_{\wf}\vf$ is $\wf$-almost dominated, this implies $\mathbf{D}_{\wf}\uf\leq^{op}\mathbf{D}_{\wf}(\mathbf{D}_{\wf}\vf)=\mathbf{D}_{\wf}\vf$. Again, by \eqref{F: t:w<s:w iff Dwt<Dws} we conclude that \begin{align*}\alpha(\uf)=\uf:\wf\geq\vf:\wf.
\end{align*}
The closed elements in $\fpx$ are those forms such that $$\tf=(\beta\circ\alpha)(\tf)=(\tf:\wf)\div\wf=\dwt.$$ That is, $\tf$ is $\beta\circ\alpha$-closed if and only if $\tf$ is $\wf$-almost dominated. Finally, we remark that $\tf:\wf=\dwt:\wf$ in \eqref{F: t:w=dwt:w} says on the one hand that $\alpha$ is not a bijection between $\fpx$ and $\fpx^{:\wf}$. Indeed, for any form $\tf\in\fpx$, $\alpha(\tf)=\alpha(\uf)$ holds for all $\uf\in[\dwt,\tf]$. On the other hand, it says that the restriction of $\alpha$ to the set $\set{\dwt}{\tf\in\fpx}$ is still surjective.\\ 

Summarizing these observations, we have the following theorem.
\begin{theorem}\label{T: Galois}
Let $\wf\in\fpx$ be a fixed form on $\Xx$, and consider the partially ordered sets $\big(\fpx,\leq^{op}\big)$ and $\big(\fpx^{:\wf},\leq\big)$. Then the $(\alpha,\beta)$ pair of order reversing functions
$$\alpha:\fpx\to\fpx^{:\wf};~\alpha(\tf)=\tf:\wf\qquad\mbox{and}\qquad\beta:\fpx^{:\wf}\to\fpx;~\beta(\ssf)=\ssf\div\wf$$
establishes an antitone Galois connection. The set of $(\beta\circ\alpha)$-closed elements is identical with the set of $\wf$-almost dominated forms. Moreover, $\alpha$ is a bijection between $\set{\dwt}{\tf\in\fpx}$ and $\set{\tf:\wf}{\tf\in\fpx}$, and its inverse is $\beta$. 
\end{theorem}
In the last part of this paper we are going to prove directly that the set of $\wf$-quasi-units is a lattice. The main advantage of this new proof is that we will determine the concrete value of the greatest lower- and least upper bound. We remark that this proof could be done without using any operator theoretic results from the previous sections. Indeed, although we are going to use that a form $\tf\in[\nullf,\wf]$ is a $\wf$-quasi-unit if and only if $\tf$ and $\wf-\tf$ are singular, and this statement is a consequence of Theorem \ref{T:quasiunit-operator} (i)$\Leftrightarrow$(iv), this was proved earlier by using only forms in \cite[Theorem 11]{cof}. Similarly, the fact that $\dwt$ is a $\tf$-quasi-unit for all $\wf$ follows from Corollary \ref{corollary [A]B quasi-unit}, but it can be proved by using only form methods, as well. For such a proof, we refer the reader to \cite[Corollary 1]{e-cof}. The following characterization of quasi-units is just the translation of \cite[Theorem 4.2 (i)$\Leftrightarrow$(iii) and Corollary 4.3]{EL} to the language of forms.

\begin{lemma}\label{2(t:w)=t}
Let $\tf$ and $\wf$ be forms on $\Xx$.  Then  $\tf$ is a $\wf$-quasi-unit if and only if $\tf:\wf=\frac{1}{2}\tf$.

\end{lemma}
\begin{proof}
If $\tf$ is a $\wf$-quasi-unit then $\tf:\wf=\tf:\dtw=\tf:\tf=\frac{1}{2}\tf$ according to \eqref{F: t:w=dwt:w}. To prove the converse implication, assume that $\tf:\wf=\frac{1}{2}\tf$. Define the sequence $\lambda_1=1$, $\lambda_k:=\lambda_{k-1}(\lambda_{k-1}+2)$ for $k\geq2$. Since $\lambda_k\geq 2^{k-1}$, it is enough to show that
\begin{align}\label{indukcio}
(\lambda_k\tf):\wf=\frac{\lambda_k}{1+\lambda_k}\tf
\end{align}
holds for all $k\in\dupN$. If this is the case, taking the supremum in $k\in\dupN$ we obtain $\dtw=\tf$. 
If $k=1$ then $\tf:\wf=\frac{1}{2}\tf$ holds by assumption. Now assume that \eqref{indukcio} holds for some $k\in\dupN$. Using the hypothesis twice, we obtain that
\begin{equation*}
\begin{split}
\frac{1}{1+\lk}\tf&=\tf:\frac{1}{\lk}\wf=\left(\frac{1+\lk}{\lk}\big((\lk\tf):\wf\big)\right):\frac{1}{\lk}\wf=\\
&=\left((1+\lk)\tf\right):\left(\frac{1+\lk}{\lk}\wf:\frac{1}{\lk}\wf\right)=\big((1+\lk)\tf\big):\frac{1+\lk}{\lk(\lk+2)}\wf=\\
&=(1+\lk)\left(\tf:\frac{1}{\lambda_{k+1}}\wf\right)=\frac{1+\lk}{\lkp}\big((\lkp\tf):\wf\big)
\end{split}
\end{equation*}
and hence,
\begin{equation*}
(\lkp\tf):\wf=\frac{\lkp}{(1+\lk)^2}\tf=\frac{\lkp}{1+\lkp}\tf.
\end{equation*}
\end{proof}

\begin{theorem}
Let $\wf$ be a form, and consider the set $\quasi(\wf)$ of $\wf$-quasi-units. Then the partially ordered set $(\quasi(\wf),\leq)$ is a lattice. Namely, if $\ssf$ and $\tf$ are $\wf$-quasi-units, then the greatest lower bound $\ssf\curlywedge\tf$ and the least upper bound $\ssf\curlyvee\tf$ in $\quasi(\wf)$ exist, and \begin{equation*} \ssf\curlywedge\tf=2(\ssf:\tf),\qquad\qquad\ssf\curlyvee\tf=\mathbf{D}_{\ssf+\tf}\wf.
\end{equation*}
Furthermore, $\ssf\curlywedge\tf=\ssf\wedge\tf=\mathbf{D}_{\ssf}\tf=
\mathbf{D}_{\tf}\ssf$.
\end{theorem}
\begin{proof}
First observe that $2(\ssf:\tf)$ is a $\wf$-quasi-unit according to the previous lemma. Indeed,
\begin{equation*}
[2(\ssf:\tf)]:\wf=[2(\ssf:\tf)]:[2(\wf:\wf)]=2[(\ssf:\wf):(\tf:\wf)]=2[(\frac{1}{2}\ssf):(\frac{1}{2}\tf)]=\ssf:\tf.
\end{equation*}
Now let $\uf$ be a form such that $\uf\leq\ssf$ and $\uf\leq\tf$. According to
\begin{equation*}
2(\ssf:\tf)\geq 2(\uf:\uf)=\uf,
\end{equation*}
it is enough to show that $2(\ssf:\tf)\leq\ssf$ and $2(\ssf:\tf)\leq\tf$. This follows immediately from the previous lemma, because
$2(\ssf:\tf)\leq2(\wf:\tf)=\tf$ and $2(\ssf:\tf)\leq2(\ssf:\wf)=\ssf$. This shows also that $2(\ssf:\tf)$ is the infimum of $\ssf$ and $\tf$ in $\fpx$.
Now we show that $2(\ssf:\tf)=\mathbf{D}_{\ssf}\tf=
\mathbf{D}_{\tf}\ssf$.
Since $\dts\leq\dtw=\tf$ we have $\dts=\mathbf{D}_{\ssf}(\dts)\leq\dst$. And by symmetry, $\dst\leq\dts$. On the other hand, $\dst\leq\tf$ and $\dts\leq\ssf$ imply $\dst:\dtw\leq\tf:\ssf$. Since $\dst=\dts$ we conclude that $\dst=\dts\leq 2(\tf:\ssf)$. For the converse inequality observe that $2(\ssf:\tf)\leq2(\ssf:\wf)=\ssf\leq\wf$, and hence $2(\ssf:\tf)=\mathbf{D}_{\tf}(2(\ssf:\tf))\leq\dtw$.

To prove that the least upper bound of $\ssf$ and $\tf$ in $\quasi(\wf)$ is $\mathbf{D}_{\ssf+\tf}\wf$ take an $\uf\in \quasi(\wf)$ such that $\uf\geq\ssf$ and $\uf\geq\tf$ and observe that $\uf=\mathbf{D}_{\uf}\wf=\mathbf{D}_{2\uf}\wf
\geq\mathbf{D}_{\ssf+\tf}\wf$. Since $\mathbf{D}_{\ssf+\tf}\wf$ is a $\wf$-quasi-unit according to \eqref{F: dwt quasi-unit}, the proof is complete by the inequalities $\mathbf{D}_{\ssf+\tf}\wf\geq\mathbf{D}_{\ssf}\wf=\ssf$ and $\mathbf{D}_{\ssf+\tf}\wf\geq\mathbf{D}_{\tf}\wf\geq\tf$.
\end{proof}

\end{document}